\documentclass[reqno]{amsart}
\usepackage[dvipsnames]{xcolor}
\usepackage{amssymb,amscd,bbm,verbatim,graphicx,mathtools,enumitem}
\usepackage{textcomp}
\usepackage[mathscr]{euscript}
\usepackage[colorlinks]{hyperref}
\newtheorem{thm}{Theorem}[section]
\newtheorem{cor}[thm]{Corollary}
\newtheorem{lem}[thm]{Lemma}

\theoremstyle{definition}
\newtheorem{defn}[thm]{Definition}

\newtheorem{hyp}[thm]{Hypothesis}

\numberwithin{equation}{section}

\newcommand{\vertiii}[1]{{\left\vert\kern-0.25ex\left\vert\kern-0.25ex\left\vert #1\right\vert\kern-0.25ex\right\vert\kern-0.25ex\right\vert}}

\newcommand{\BVl}{\operatorname{BV}_{\rm loc}}

\newcommand{\id}{\mathbbm 1}

\newcommand{\loc}{{\rm loc}}

\newcommand{\iOmega}{(a,b)}

\begin{document}
\title[]{sign-changing points of solutions of homogeneous sturm-liouville equations with measure-valued coefficients}%
\author{ahmed ghatasheh and rudi weikard}%

\address{A.G.: mathematics department, the ohio state university at marion, marion, oh 43302, usa}
\email{ghatasheh.1@osu.edu}%
\address{R.W.: department of mathematics, university of alabama at birmingham, birmingham, al 35226-1170, usa}%
\email{weikard@uab.edu}%

\date{\today}%
\thanks{Name of \TeX{} file: \texttt{ZS1.tex}}%



\begin{abstract}
In this paper we investigate sign-changing points of nontrivial real-valued solutions of homogeneous Sturm-Liouville differential equations of the form $-d(du/d\alpha)+ud\beta=0$, where $d\alpha$ is a positive Borel measure supported everywhere on $\iOmega$ and $d\beta$ is a locally finite real Borel measure on $\iOmega$.
Since solutions for such equations are functions of locally bounded variation, sign-changing points are the natural generalization of zeros. We prove that sign-changing points for each nontrivial real-valued solution are isolated in $\iOmega$. We also prove a Sturm-type separation theorem for two nontrivial linearly independent solutions, and conclude the paper by proving a Sturm-type comparison theorem for two differential equations with distinct potentials.
\end{abstract}

\maketitle

\section{introduction}

When $1/p$ and $q$ are real-valued locally integrable functions on $\iOmega$, $-\infty\leq a<b\leq\infty$, and when $p>0$ on $\iOmega$, zeros of every nontrivial real-valued solution of the famous homogeneous Sturm-Liouville differential equation
\begin{equation}\label{trpz8104}
-(py')'+qy=0
\end{equation}
are isolated in $\iOmega$. This fact is the key to establish two big results in Sturm-Liouville theory, namely the Sturm separation theorem and the Sturm comparison theorem. These two celebrated results are due to Sturm \cite{Sturm1836a} and they date back to 1836. The Sturm separation theorem states that zeros of two linearly independent solutions of Equation \eqref{trpz8104} are interlaced. The Sturm comparison theorem states that if $u$ and $v$ are nontrivial solutions of \eqref{trpz8104} and $-(py')'+\tilde{q}y=0$, respectively, $u$ vanishes at $s$ and $t$, and $\tilde q\leq q$ but different on a set of positive Lebesgue measure, then $v$ vanishes at some point between $s$ and $t$.

In this paper we consider homogeneous differential equations of the form
\begin{equation}\label{sgbzpie9573}
-d\Big(\frac{dy}{d\alpha}\Big)+yd\beta=0
\end{equation}
when $d\alpha$ is a positive measure supported everywhere on $\iOmega$ and when $d\beta$ is a locally finite real Borel measure on $\iOmega$. In other words $\alpha$ is a strictly increasing function on $\iOmega$ and $\beta$ is a real-valued function of locally bounded variation on $\iOmega$. Solutions for Equation \eqref{sgbzpie9573} are functions of locally bounded variation on $\iOmega$. For any nontrivial real-valued solution $u$ of Equation \eqref{sgbzpie9573} it is possible to have $u^-(x_0)=0$, $u^+(x_0)=0$, or $u^-(x_0)u^+(x_0)<0$ for some $x_0\in\iOmega$, where $u^-$ and $u^+$ are the left- and right-continuous representatives of $u$. The set of all such points, which we denote by $\mathcal{Z}_{\iOmega}(u)$, is the natural candidate for zeros when we step up from locally integrable coefficient (Equation \eqref{trpz8104}) to measure valued coefficients (Equation \eqref{sgbzpie9573}). The first question that comes to mind is whether $\mathcal{Z}_{\iOmega}(u)$ is a set of isolated points. If so, then the next question is whether we can develop a Sturm-type separation and comparison theorems for nontrivial real-valued solutions of equations of the form \eqref{sgbzpie9573}.

The goal of the present work is to answer the above questions. As far as we know this is the first work that investigates such questions when the measure-valued coefficients in the equation are chosen so that solutions are of locally bounded variation.

In 1909 Picone \cite{Picone1910} discovered his famous identity that simplified the proof of Sturm's comparison theorem for two equations of the form \eqref{trpz8104}. Another essential improvement is due Leighton \cite{MR0140759} who pointed out that it suffices to have an integral condition instead of the Sturm's pointwise condition. Mingarelli \cite{MR706255} proved a Leighton-type comparison theorem when only the potential $q$ in Equation \eqref{trpz8104} is a locally finite real Borel measure. Also Ghatasheh and Weikard \cite{MR3624546} provided a Leighton-type comparison theorem that covers Sturm-Liouville equations with distributional potentials. All these works rely on the absolute continuity of solutions which is not valid for equations of the form \eqref{sgbzpie9573}. There are excellent references concerned with comparison theorems and their historical survey, we only mention Hinton \cite{zbMATH02247495}, Mingarelli \cite{MR706255}, Swanson \cite{zbMATH03303390}, and Zettl \cite{MR2170950} and the references therein.

The first to consider differential equations with measure-valued coefficients were Krein \cite{MR0054078} and Feller \cite{MR63607}.
The interest in Sturm-Liouville differential equations with measure-valued coefficients increased recently. Perhaps the wish to have a unified approach for the discrete and continuous cases motivated both Atkinson \cite{MR0176141} and Mingarelli \cite{MR706255}. The basic assumptions in Chapter 3 of \cite{MR706255} require the generating functions for the considered measures to be continuous at boundary points. Recently Eckhardt and Teschl \cite{MR3095152} relaxed Mingarelli's basic assumptions by assuming that the measures do not jump at the same time. More recently Ghatasheh and Weikard \cite{GW18-1} have an essential improvement by allowing measures to jump at the same time by considering balanced solutions, they are functions of locally bounded variations whose values are the average of left- and right-hand limits. There are many other papers considering measure-valued coefficients, see for example Bennewitz \cite{MR1004432}, Persson \cite{MR961581}, and Volkmer \cite{MR2135259} and the references therein.

Differential equations with measure-valued coefficients are important for many reasons. First of all, they are a unification for both the discrete and the continuous cases. If we choose $d\alpha$ and $d\beta$ to be discrete measures in \eqref{sgbzpie9573}, then we obtain homogeneous difference equations. On the other hand, if we choose $\alpha$ and $\beta$ to be antiderivatives of $1/p$ and $q$, respectively, then we obtain Equation \eqref{trpz8104}. Also, differential equations with measure-valued coefficients cover important physical models such as $\delta$ and $\delta'$ interactions; for interested readers we mention Gesztesy and Holden \cite{MR914699}, Albeverio et al. \cite{MR926273}, and Eckhardt et al. \cite{MR3200376}.

In Section \ref{pr0pr} we introduce the definitions and notations that we use in this paper. In Section \ref{eus0eus} we discuss briefly the meaning of Sturm-Liouville differential equations with measure-valued coefficients as well as the associated existence and uniqueness theorems. In Section \ref{scps0scps} we investigate the behavior of nontrivial solutions $u$ of Equation \eqref{sgbzpie9573} at points in $\mathcal{Z}_{\iOmega}(u)$. In Section \ref{spt0spt} we prove a Sturm-type separation theorem for linearly independent solutions of Equation \eqref{sgbzpie9573}. Finally, in section \ref{sct0sct}, we prove a Sturm-type comparison theorem for only balanced solutions of two equations of the form \eqref{sgbzpie9573} with the same $d\alpha$ and two distinct potentials, $d\beta_1$ and $d\beta_2$.

\section{preliminaries}\label{pr0pr}

We denote the space of all complex-valued functions that are of locally bounded variation on $\iOmega$ by $\BVl(\iOmega,\mathbb{C})$. The subset of $\BVl(\iOmega,\mathbb{C})$ consisting of real-valued functions is denoted by $\BVl(\iOmega,\mathbb{R})$. The corresponding left- and right-continuous representatives of a function $f\in\BVl(\iOmega,\mathbb{C})$ are defined by setting $f^-(x)=\lim_{t\uparrow x}f(t)$ and $f^+(x)=\lim_{t\downarrow x}f(t)$, respectively. Given $r\in\mathbb{C}$, we say that a function $f$ is $r$-balanced if $f\in\BVl(\iOmega,\mathbb{C})$ and $f(x)=rf^-(x)+(1-r)f^+(x)$ for all $x\in\iOmega$. We denote the space of all $r$-balanced complex-valued functions on $\iOmega$ by $\BVl^{r}(\iOmega,\mathbb{C})$ and the subspace consisting of those that are real-valued by $\BVl^{r}(\iOmega,\mathbb{R})$. When $r=1/2$ we simply say balanced instead of $1/2$-balanced. These definitions and notations can be extended to matrix-valued functions. For example we denote by $\BVl^r(\iOmega,\mathbb{C}^{m\times n})$ the space of all functions $f:\iOmega\rightarrow\mathbb{C}^{m\times n}$ for which each component of $f$ is $r$-balanced.

The locally finite complex Borel measure generated by $f\in\BVl(\iOmega,\mathbb{C})$ is denoted by $df$. We denote the measure of a singleton $\{s\}$ (i.e., $df(\{s\})$) by $\Delta_{df}(s)$. Note that $\Delta_{df}(s)=f^+(s)-f^-(s)$. We denote the total variation of $df$ by $df_{\uparrow}$. Given $h\in L^1_{\loc}\big(\iOmega,df_{\uparrow})$, the locally finite complex Borel measure $hdf$ is defined to be $dh_0$ where
$$h_0(x)=\begin{cases}
      \;\;\;\int_{[c,x)}h\;df, & x\geq c\\
      -\int_{[x,c)}h\;df, & x<c.
   \end{cases}$$

If $f,g\in\BVl(\iOmega,\mathbb{C})$ satisfy $df\ll dg$ (i.e., $df$ is locally absolutely continuous with respect to $dg_{\uparrow}$), then the Radon-Nikodym derivative of $df$ with respect to $dg$ is denoted by $df/dg$. It is uniquely defined in
$L^1_{\loc}(\iOmega,dg_{\uparrow})$.

The product rule for any $f,g\in\BVl(\iOmega,\mathbb{C})$ (see Hewitt and Stromberg \cite{MR0188387}, Theorem 21.67 and Remark 21.68)) is given by
$$d(fg)=f^-\;dg+g^+\;df=f^+\;dg+g^-\;df.$$
If $f$ and $g$ are balanced functions, then
\begin{equation}\label{erybs0341}
d(fg)=f\;dg+g\;df.
\end{equation}

The support of the locally finite complex Borel measure $df$ can be defined by setting $x\not\in\text{supp}(df)$ if there exists $\epsilon>0$ such that $[x-\epsilon,x+\epsilon]\subset\iOmega$ and $(x-\epsilon,x+\epsilon)$ is a $df$-null set (i.e., any Borel subset of $(x-\epsilon,x+\epsilon)$ has $df$-measure zero).

\section{existence and uniqueness of solutions}\label{eus0eus}

We begin by investigating initial value problems of the form
\begin{equation}\label{A21212001}
-d\Big(\frac{du}{d\alpha}\Big)+u\;d\beta=d\eta,\;\;u(x_0)=A,\;\;\frac{du}{d\alpha}(x_0)=B,
\end{equation}
where $\alpha,\beta,\eta\in\BVl(\iOmega,\mathbb{C})$. We seek solutions to this equation among functions of locally bounded variation, since for any such function $u$ the terms $ud\beta$ and $d\eta$ are locally finite complex Borel measures on $\iOmega$ and if, in addition, $du$ is absolutely continuous with respect to $d\alpha$ and the Radon-Nikodym derivative $du/d\alpha$ has at least one representative of locally bounded variation, then $d(du/d\alpha)$ is a locally finite complex Borel measure on $\iOmega$ so that it makes sense to pose the equation.

The linear system corresponding to the initial value problem \eqref{A21212001} is given by
\begin{equation}\label{A12011}
dv=d\alpha_0\;v+ d\eta_0,\;v(x_0)=A_0
\end{equation}
where
$$v=\left( {\begin{array}{cc}
   u\\
   \frac{du}{d\alpha}\\
  \end{array} } \right),\;
\alpha_0=\left( {\begin{array}{cc}
   0 & \alpha\\
   \beta & 0\\
  \end{array} } \right),\;
  \eta_0=\left( {\begin{array}{cc}
   0\\
   \eta\\
  \end{array} } \right),\;\text{and }
  A_0=\left( {\begin{array}{cc}
   A\\
   B\\
  \end{array} } \right).$$

Existence and uniqueness theorems of left- and right-continuous solution of locally bounded variation (the cases $r=1$ and $r=0$, respectively) for such linear systems were discussed in many papers, see for example Theorem 1.1 in Bennewitz \cite{MR1004432} or Theorem 3.1 in Eckhardt and Teschl \cite{MR3095152}. For the existence and uniqueness of balanced solutions (the case $r=1/2$) see Theorem 2.2 in Ghatasheh and Weikard \cite{GW18-1}. The following theorem provides necessary conditions for existence and uniqueness of $r$-balanced solutions. The proof is omitted since it is similar to the proof of Theorem 2.2 in \cite{GW18-1}. We denote the identity $n\times n$ matrix by $\id$ and the determinant of a square matrix $C$ by $\det(C)$.
\begin{thm}\label{th01} Let $r\in\mathbb{C}$, $\phi\in\BVl(\iOmega,\mathbb{C}^{n\times n})$, and $\psi\in\BVl(\iOmega,\mathbb{C}^{n})$.
If $\det\big(\id-(1-r)\Delta_{d\phi}(x)\big)\neq0\;
\text{and}\;\det\big(\id+r\Delta_{d\phi}(x)\big)\neq0\;\text{for all}\;x\in\iOmega$,
then any initial value problem of the form
$$dy=d\phi\;y+d\psi,\;\;y(x_0)=y_0,$$
where $x_0\in\iOmega$ and $y_0\in\mathbb{C}^n$, has a unique solution in $\BVl^r(\iOmega,\mathbb{C}^{n})$. If, in addition,
$r$, $\phi$, $\psi$, and $y_0$ are real, then the solution is real.
\end{thm}

According to Theorem \ref{th01} if $(1-r)^2\Delta_{d\alpha}(x)\Delta_{d\beta}(x)\neq1$ and $r^2\Delta_{d\alpha}(x)\Delta_{d\beta}(x)\neq1$ for all $x\in\iOmega$, then system \eqref{A12011} has a unique solution $v\in\BVl^r(\iOmega,\mathbb{C}^{2})$. This means the components of $v$ satisfy $dv_1=v_2d\alpha$ and $dv_2=v_1d\beta+d\eta$. The first equation is equivalent to $v_2=dv_1/d\alpha$ $d\alpha$-a.e. on $\iOmega$. When we specialize to the initial value problem \eqref{A21212001} we want $v_2=dv_1/d\alpha$ everywhere on $\iOmega$. The following theorem shows that this can be achieved by assuming that the support of $d\alpha$ is all of $\iOmega$.
\begin{thm}\label{A1}
Let $\alpha\in\BVl(\iOmega,\mathbb{C})$ with $\text{supp}(d\alpha)=\iOmega$. If $f$ is an $r$-balanced function that satisfies $f=0$, $d\alpha$-a.e. on $\iOmega$, then $f$ is identically zero on $\iOmega$.
\end{thm}

\begin{proof}
Assume to the contrary that $f(x_1)\neq0$ for some $x_1\in\iOmega$. Since $f$ is $r$-balanced, $f^-(x_1)\neq0$ or $f^+(x_1)\neq0$. For simplicity we only consider the case $f^-(x_1)\neq0$. Since $f^-$ is left-continuous on $\iOmega$, there exists $x_0\in(a,x_1)$ such that $f(x)\neq0$ on $(x_0,x_1)$. This implies that $(x_0,x_1)$ is a $df$-null set. This contradiction completes the proof.
\end{proof}

The minimal set of assumptions that we have to make to assure the existence and uniqueness of $r$-balanced solutions of initial value problems of the form \eqref{A21212001} are presented in the following hypothesis. For any $z\in\mathbb{C}$, we define $\theta_z(x)=1-z^2\Delta_{d\alpha}(x)\Delta_{d\beta}(x)$ for all $x\in\iOmega$.

\begin{hyp}\label{sop097ts}
$r$ is a fixed complex number, $\alpha$ and $\beta$ belong to $\BVl(\iOmega,\mathbb{C})$, the support of $d\alpha$ is $\iOmega$, and $\theta_{1-r}(x)\neq0$ and $\theta_{r}(x)\neq0$ for all $x\in\iOmega$.
\end{hyp}

\begin{defn}
Assume that $\eta\in\BVl(\iOmega,\mathbb{C})$ and $r$, $\alpha$, and $\beta$ satisfy Hypothesis \ref{sop097ts}. $u$ is said to be an $r$-balanced solution of the differential equation $-d(dy/d\alpha)+yd\beta=d\eta$ on $\iOmega$ if $u$ is $r$-balanced on $\iOmega$, $du\ll d\alpha$, $du/d\alpha$ is $r$-balanced on $\iOmega$, and $-d(du/d\alpha)+ud\beta=d\eta$. $u$ is called balanced solution if $r=1/2$.
\end{defn}

Now we are ready to state the existence and uniqueness theorem for initial value problems of the form \eqref{A21212001}.

\begin{thm}\label{th00121}
Assume that $\eta\in\BVl(\iOmega,\mathbb{C})$ and $r$, $\alpha$, and $\beta$ satisfy Hypothesis \ref{sop097ts}. Then any initial value problem of the form
\begin{equation}
-d\Big(\frac{dy}{d\alpha}\Big)+y\;d\beta=d\eta,\;\;y(x_0)=A,\;\;\frac{dy}{d\alpha}(x_0)=B,
\end{equation}
where $A,B\in\mathbb{C}$ and $x_0\in\iOmega$, has a unique $r$-balanced solution. If, in addition, $r$, $\alpha$, $\beta$, $\eta$, $A$ and $B$ are real, then the solution is real.
\end{thm}

We close this section by deriving two equations that will be used extensively in the coming sections. We begin by integrating Equation \eqref{A12011} at each singleton, when $d\eta=0$. We obtain $v^+-v^-=\Delta_{d\alpha_0}v$. Also since $v$ is $r$-balanced we have $rv^-+(1-r)v^+=v$. By eliminating $v^-$ we obtain
\begin{equation}\label{we7610s200}
\left( {\begin{array}{cc}
   u^+\\
   (\frac{du}{d\alpha})^+\\
  \end{array} } \right)=\left( {\begin{array}{cc}
   1 & r\Delta_{d\alpha}\\
   r\Delta_{d\beta} & 1\\
  \end{array} } \right)\left( {\begin{array}{cc}
  u\\
   \frac{du}{d\alpha}\\
  \end{array} } \right)
\end{equation}
and by eliminating $v^+$ we obtain
\begin{equation}\label{wg9sur0165}
\left( {\begin{array}{cc}
u^-\\
(\frac{du}{d\alpha})^-\\
\end{array} } \right)=\left( {\begin{array}{cc}
1 & -(1-r)\Delta_{d\alpha}\\
-(1-r)\Delta_{d\beta} & 1\\
\end{array} } \right)\left( {\begin{array}{cc}
u\\
\frac{du}{d\alpha}\\
\end{array} } \right).
\end{equation}
Also observe that in each one of these equations we can solve for
$$\left( {\begin{array}{cc}
  u\\
   \frac{du}{d\alpha}\\
  \end{array} } \right)$$
since the hypothesis in Theorem \ref{th00121} assures the invertibility of each of the $2\times 2$ matrices in Equations \eqref{we7610s200} and \eqref{wg9sur0165}.

\section{sign-changing points of solutions}\label{scps0scps}

When $q$ and $1/p$ are locally integrable on $\iOmega$ solutions of the differential equation $-(pu')'+qu=0$ are locally absolutely continuous on $\iOmega$. If $p>0$ and $q$ is real-valued on $\iOmega$, then zeros of nontrivial real-valued solutions are isolated in $\iOmega$. Moreover, every nontrivial solution changes sign at each of its zeros. The situation is more complicated for the differential equation $-d(du/d\alpha)+ud\beta=0$ since solutions are functions of locally bounded variation and that includes the possibility of having a countable number of discontinuity points in $\iOmega$. This section is devoted to develop a theory for the behavior of nontrivial solutions $u$ at points $s\in\iOmega$ that satisfy $u^-(s)u^+(s)\leq0$.

Let $f\in\BVl(\iOmega,\mathbb{R})$ and let $I$ be a subinterval of $\iOmega$. We denote by $\mathcal{Z}_{I}(f)$ the set of all $x\in I$ such that $f^-(x)f^+(x)\leq0$.

\begin{thm}[Intermediate Value Theorem]\label{ss3}
Let $f\in\BVl(\iOmega,\mathbb{R})$. If $I$ is a subinterval of $\iOmega$ such that $\mathcal{Z}_{I}(f)$ is empty, then either $f^-,f^+>0$ on $I$ or $f^-,f^+<0$ on $I$. Furthermore, $f^-$ and $f^+$ are bounded away from zero on $I$ if $I$ is compact.
\end{thm}

\begin{proof}
Since $\mathcal{Z}_{I}(f)$ is empty, it is easy to check that $f^-<0$ on $I$ if and only if $f^+<0$ on I. Also $f^->0$ on I if and only if $f^+>0$ on $I$. To prove the first part of the theorem, assume to the contrary that there are $s,t$ in $I$ such that $f^-(s)f^-(t)\leq0$. We may exclude the cases $f^-(s)f^-(t)=0$ and $s=t$ since they lead to a direct contradiction. Let us assume that $s<t$ in $I$, $f^-(s)>0$, and $f^-(t)<0$. Then $X=\{x\in[s,t]:f^->0\;\text{on}\;[s,x]\}$ is a nonempty subset of $\iOmega$ that is bounded above by $t$. Since $f^-$ is left-continuous and $\mathcal{Z}_{I}(f)$ is empty, the least upper bound of $X$, say $x_0$, satisfies $f^-(x_0)>0$. Consequently $f^+(x_0)>0$ and so for some $\delta, M>0$, $f^+(x)\geq M$ on $(x_0,x_0+\delta)$. Then $f^-(x)\geq M$ on $(x_0,x_0+\delta)$ which contradicts $x_0$ is the least upper bound of $X$. Since a similar argument works when $f^-(s)<0$ and $f^-(t)>0$ the proof of the first part is complete.

The second part of the theorem follows from the fact that the least upper bound of $f^-$ on a compact subinterval $J$ of $\iOmega$ is attained at a point of the form $f^{\mp}(z)$, where $z\in J$. The same statement holds for the greatest lower bound.
\end{proof}

If $f\in\BVl(\iOmega,\mathbb{R})$ such that $\mathcal{Z}_{\iOmega}(f)$ is empty, then the Intermediate Value Theorem \ref{ss3} tells us that for each $r\in[0,1]$, either $rf^-+(1-r)f^+>0$ on $\iOmega$ or $rf^-+(1-r)f^+<0$ on $\iOmega$. Furthermore it is bounded away from zero on each compact subinterval of $\iOmega$.

\begin{cor}\label{es54uaw}
Let $f\in\BVl(\iOmega,\mathbb{R})$. If $y_0$ is a number between $f^-(s)$ and $f^-(t)$, exclusive, for some interval $[s,t]\subset\iOmega$, then $y_0$ is between $f^-(x_0)$ and $f^+(x_0)$, inclusive, for some $x_0\in[s,t]$.
\end{cor}

\begin{proof}
It follows from the Intermediate Value Theorem \ref{ss3} that $\mathcal{Z}_{[s,t]}(f-y_0)$ is not empty, which means $(f^-(x_0)-y_0)(f^+(x_0)-y_0)\leq0$ for some $x_0\in[s,t]$. This in turn implies that $y_0$ is between $f^-(x_0)$ and $f^+(x_0)$, inclusive.
\end{proof}

A similar version of Corollary \ref{es54uaw} is valid if $y_0$ is between
$f^-(s)$ and $f^+(t)$, $f^+(s)$ and $f^-(t)$, or $f^+(s)$ and $f^+(t)$.

\begin{thm}[Mean Value Theorem]\label{sss05}
Let $r$ be a fixed number in $[0,1]$. Let $\alpha$ be a strictly increasing function on $\iOmega$ and let $u$ be a real-valued $r$-balanced function on $\iOmega$ such that $du\ll d\alpha$ and $du/d\alpha$ is $r$-balanced on $\iOmega$. If $s<t$ are in $\mathcal{Z}_{\iOmega}(u)$, then $\mathcal{Z}_{[s,t]}(du/d\alpha)$ is not empty.
\end{thm}

\begin{proof}
Assume to the contrary that $\mathcal{Z}_{[s,t]}(du/d\alpha)$ is empty. It follows from the Intermediate Value Theorem \ref{ss3} that either $du/d\alpha>0$ on $[s,t]$ or $du/d\alpha<0$ on $[s,t]$. Let us assume that $du/d\alpha>0$ on $[s,t]$. For any interval $I\subset[s,t]$ that contains more than one point
$$\int_I du=\int_I \frac{du}{d\alpha}\;d\alpha>0$$
since $\alpha$ is strictly increasing on $\iOmega$ and $du/d\alpha>0$ on $[s,t]$. By choosing $I$ to be $(s,t)$, $(s,t]$, $[s,t)$, and $[s,t]$ we obtain $u^-(t)-u^+(s)>0$, $u^+(t)-u^+(s)>0$, $u^-(t)-u^-(s)>0$, and $u^+(t)-u^-(s)>0$, respectively. The first two inequalities imply $u^+(s)<\min\{u^-(t),u^+(t)\}$ and the second two inequalities imply $u^-(s)<\min\{u^-(t),u^+(t)\}$. But $\min\{u^-(t),u^+(t)\}\leq0$, so $u^-(s)$ and $u^+(s)$ are both negative and this contradicts the assumption that $s\in\mathcal{Z}_{\iOmega}(u)$. Therefore $\mathcal{Z}_{[s,t]}(du/d\alpha)$ is not empty. Since a similar argument works when $du/d\alpha<0$ on $[s,t]$ the proof is complete.
\end{proof}

We point out that it suffices to assume that the support of $d\alpha$ is $\iOmega$ and $\alpha$ is strictly increasing on some open interval containing $[s,t]$ to have the conclusion of the Mean Value Theorem \ref{sss05}. Furthermore, the theorem can be strengthened given the exact behavior of $u$ at $s$ and $t$. For example if $u^+(s)=u^-(t)=0$, then $\mathcal{Z}_{(s,t)}(du/d\alpha)$ is not empty (note that we now have an open interval $(s,t)$).

In most of the coming results we require the following hypothesis to be satisfied.

\begin{hyp}\label{H:4.1}
$r$ is a fixed number in $[0,1]$, $\alpha$ is strictly increasing on $\iOmega$, $\beta\in\BVl(\iOmega,\mathbb{R})$, and $\theta_{1-r}(x)\neq0$ and $\theta_{r}(x)\neq0$ for all $x\in\iOmega$.
\end{hyp}

Now we are ready to prove the main result of this section.

\begin{thm}\label{10664yrlz}
Assume that $r$, $\alpha$, and $\beta$ satisfy Hypothesis \ref{H:4.1}. If $u$ is a nontrivial real-valued $r$-balanced solution of $-d(dy/d\alpha)+yd\beta=0$ on $\iOmega$, then $\mathcal{Z}_{\iOmega}(u)$ is a set of isolated points.
\end{thm}

\begin{proof}
Let $u$ be a real-valued $r$-balanced solution of $-d(dy/d\alpha)+yd\beta=0$ on $\iOmega$ such that $\mathcal{Z}_{\iOmega}(u)$ contains a point $s$ that is not isolated. We are going to show that $u$ is the trivial solution. Since $s$ is not isolated there is a sequence of distinct points $x_n\in\mathcal{Z}_{\iOmega}(u)$ that converges to $s$. We may assume that either $x_n>s$ for all $n$ or $x_n<s$ for all $n$. Let us consider the case $x_n>s$ for all $n$. According to the Mean Value Theorem \ref{sss05} there exists a sequence $y_n\in\mathcal{Z}_{[s,x_n]}(du/d\alpha)$. By taking the limit as $n\rightarrow\infty$ to each of the inequalities $u^-(x_n)u^+(x_n)\leq0$ and $(du/d\alpha)^-(y_n)(du/d\alpha)^+(y_n)\leq0$ we obtain $u^+(s)=(du/d\alpha)^+(s)=0$. According to Equation \eqref{we7610s200} $u(s)=(du/d\alpha)(s)=0$. According to the uniqueness theorem for initial value problems $u$ is the trivial solution. Since a similar argument works for the case $x_n<s$ for all $n$ the proof is complete.
\end{proof}

\begin{defn}
Assume that $r$, $\alpha$, and $\beta$ satisfy Hypothesis \ref{H:4.1}. Let $u$ be a nontrivial real-valued $r$-balanced solution of $-d(dy/d\alpha)+yd\beta=0$ on $\iOmega$. $u$ is said to change sign at $s\in\iOmega$ if there exists $\delta>0$ such that $u^-<0$ on $(s-\delta,s)$ and $u^+>0$ on $(s,s+\delta)$, or $u^->0$ on $(s-\delta,s)$ and $u^+<0$ on $(s,s+\delta)$. $u$ is said to change sign in a subinterval $I$ of $\iOmega$ if it changes sign at some point in $I$.
\end{defn}

Observe that if $u$ changes sign at $s\in\iOmega$, then
$s\in\mathcal{Z}_{\iOmega}(u)$.

\begin{thm}\label{zzdnxt401}
Assume that $r$, $\alpha$, and $\beta$ satisfy Hypothesis \ref{H:4.1}. Let $u$ be a nontrivial real-valued $r$-balanced solution of $-d(dy/d\alpha)+yd\beta=0$ on $\iOmega$ and let $s\in\iOmega$. If $u(s)=0$, then $u$ changes sign at $s$.
\end{thm}

\begin{proof}
Since $u$ is not the trivial solution it follows that $(du/d\alpha)(s)\neq0$. Also since $u(s)=0$ it follows that $(du/d\alpha)^-(s)=(du/d\alpha)^+(s)=(du/d\alpha)(s)$. Then there exists $\delta>0$ such that either $du/d\alpha>0$ on $(s-\delta,s+\delta)$ or $du/d\alpha<0$ on $(s-\delta,s+\delta)$. If $x\in(s-\delta,s)$ and $y\in(s,s+\delta)$, then
$$(u^-(s)-u^-(x))(u^+(y)-u^-(s))=\int_{[x,s)}\frac{du}{d\alpha}d\alpha \int_{[s,y]}\frac{du}{d\alpha}d\alpha>0.$$
This shows that $u^-(s)$ is between $u^-(x)$ and $u^+(y)$. A similar argument, by integrating over $[x,s]$ and $(s,y]$, shows that $u^+(s)$ is between $u^-(x)$ and $u^+(y)$. Therefore $0=u(s)=ru^-(s)+(1-r)u^+(s)$ is between $u^-(x)$ and $u^+(y)$. Since $x$ and $y$ are arbitrary points in $(s-\delta,s)$ and $(s,s+\delta)$, respectively, it follows from Theorem \ref{10664yrlz} that $u$ changes sign at $s$.
\end{proof}

\begin{thm}\label{ytrop037}
Assume that $r$, $\alpha$, and $\beta$ satisfy Hypothesis \ref{H:4.1}. Let $u$ be a nontrivial real-valued $r$-balanced solution of $-d(dy/d\alpha)+yd\beta=0$ on $\iOmega$. Then $u$ does not change sign in $\iOmega$ if and only if either $u>0$ on $\iOmega$ or $u<0$ on $\iOmega$.
\end{thm}

\begin{proof}
If $u>0$ on $\iOmega$, then $u^\mp\geq0$ on $\iOmega$ and so $u$ does not change sign in $\iOmega$. A similar statement holds when $u<0$ on $\iOmega$.

To prove the converse assume to the contrary that there are $s<t$ in $\iOmega$ such that $u(s)u(t)<0$ (the equality leads to a direct contradiction according to Theorem \ref{zzdnxt401}). We may assume that $u$ is continuous at $s$ and $t$. According to the Intermediate Value Theorem \ref{ss3} $\mathcal{Z}_{(s,t)}(u)$ is not empty. With the aid of Theorem \ref{10664yrlz}
$\mathcal{Z}_{(s,t)}(u)$ is finite. We may assume that $\mathcal{Z}_{(s,t)}(u)$ consists of the points $x_1<x_2<\dots<x_p$. Put $x_0=s$ and $x_{p+1}=t$. Then $u$ have the same sign on each of $(x_{j-1},x_{j})$ and $(x_{j},x_{j+1})$ for all $1\leq j\leq p$ since $u$ does not change sign in $\iOmega$. This contradicts that $u^-(s)u^-(t)<0$.
\end{proof}

\begin{thm}\label{ertssiao0}
Assume that $r$, $\alpha$, and $\beta$ satisfy Hypothesis \ref{H:4.1}. Let $u$ be a nontrivial real-valued $r$-balanced solution of $-d(dy/d\alpha)+yd\beta=0$ on $\iOmega$ and let $s\in\iOmega$. Then
\begin{enumerate}[wide=0pt, leftmargin=19pt, labelwidth=0pt, align=left]
\item If $u^-(s)=0$, then $u$ changes sign at $s$ if and only if $\theta_{1-r}(s)>0$.
\item If $u^+(s)=0$, then $u$ changes sign at $s$ if and only if $\theta_{r}(s)>0$.
\item $u^-(s)u^+(s)<0$ if and only if $-r\Delta_{d\alpha}(s)<u(s)/(du/d\alpha)(s)<(1-r)\Delta_{d\alpha}(s)$.
\end{enumerate}
\end{thm}

\begin{proof}
The first and the second statements have similar proofs so we only prove the first statement. Assume that $u^-(s)=0$. The case $\Delta_{d\alpha}(s)=0$ is trivial since Theorem \ref{zzdnxt401} tells us that $u$ changes sign at $s$. Let us assume that $\Delta_{d\alpha}(s)\neq0$. It follows from Equations \eqref{we7610s200} and \eqref{wg9sur0165} that
$$u^+(s)=\frac{\Delta_{d\alpha}(s)}{\theta_{1-r}(s)}\;\Big(\frac{du}{d\alpha}\Big)^-(s)$$
which is not zero since $u$ is not the trivial solution. Then
$\theta_{1-r}(s)>0$ if and only if $u^+(s)(du/d\alpha)^-(s)>0$. This completes the proof since $u$ changes sign at $s$ if and only if $u^+(s)(du/d\alpha)^-(s)>0$.

The third statement follows from $u^-(s)=u(s)-(1-r)\Delta_{d\alpha}(s)(du/d\alpha)(s)$ and  $u^+(s)=u(s)+r\Delta_{d\alpha}(s)(du/d\alpha)(s)$. These two equations can be seen directly from Equations \eqref{we7610s200} and \eqref{wg9sur0165}.
\end{proof}

The following result is an immediate consequence of Theorem \ref{ertssiao0}.

\begin{cor}
Assume that $r$, $\alpha$, and $\beta$ satisfy Hypothesis \ref{H:4.1}. Then for every nontrivial real-valued $r$-balanced solution $u$ of $-d(dy/d\alpha)+yd\beta=0$ on $\iOmega$ the set $\mathcal{Z}_{\iOmega}(u)$ consists precisely of all points in $\iOmega$ at which $u$ changes sign if and only if $\theta_r(x)>0$ and $\theta_{1-r}(x)>0$ for all $x\in\iOmega$.
\end{cor}

\section{sturm separation theorem}\label{spt0spt}

Assume that $r$, $\alpha$, and $\beta$ satisfy Hypothesis \ref{H:4.1}. The Wronskian of two $r$-balanced solutions $u$ and $v$ of $-d(dy/d\alpha)+yd\beta=0$ is defined by
$$W[u,v](x)=u(x)\frac{dv}{d\alpha}(x)-v(x)\frac{du}{d\alpha}(x).$$
Obviously $W[u,v]\in\BVl(\iOmega,\mathbb{R})$. With the aid of Equations \eqref{we7610s200} and \eqref{wg9sur0165}, it is straightforward to verify that
\begin{equation}\label{ze5610rr}
W[u,v]=W^-[u,v]/\theta_{1-r}=W^+[u,v]/\theta_{r}.
\end{equation}
Then the following statements are equivalent.
\begin{enumerate}[wide=0pt, leftmargin=19pt, labelwidth=13pt, align=left]
\item $u$ and $v$ are linearly dependent.
\item $W[u,v](x)=0$ for some $x\in\iOmega$.
\item $W^-[u,v](x)=0$ for some $x\in\iOmega$.
\item $W^+[u,v](x)=0$ for some $x\in\iOmega$.
\end{enumerate}

We define $\theta(x)=\theta_r(x)/\theta_{1-r}(x)$. Then it follows from Equation \eqref{ze5610rr} that
\begin{equation}\label{zz32901}
W^+[u,v](x)=W^-[u,v](x)\theta(x),
\end{equation}

The number of points $x\in\iOmega$ for which $\theta_r(x)<0$ is finite in each compact subset of $\iOmega$. Similar statement holds for each of $\theta_{1-r}$ and $\theta$.

The following result determines the derivative of the Wronskian. It uses the fact that $\Delta_{d\alpha}d\beta=\Delta_{d\beta}d\alpha$.

\begin{thm}
Assume that $r$, $\alpha$, and $\beta$ satisfy Hypothesis \ref{H:4.1}. If $u$ and $v$ are $r$-balanced solutions of $-d(dy/d\alpha)+yd\beta=0$, then
\begin{equation}\label{strj0121}
dW[u,v]=W^-[u,v]\frac{(1-2r)\Delta_{d\alpha}}{\theta_{1-r}}d\beta=W^+[u,v]\frac{(1-2r)\Delta_{d\alpha}}{\theta_{r}}d\beta
\end{equation}
\end{thm}

For the standard Sturm-Liouville homogeneous differential equation, the Wronskian of two linearly independent solution is a nonzero constant everywhere on $\iOmega$ which is the key to prove the Sturm separation theorem, see for example Zettl \cite{MR2170950}.
Unfortunately Equation \eqref{strj0121} tells us that the derivative of $W[u,v]$, even when $u$ and $v$ are linearly independent, is not necessarily the zero measure, unless the two solutions are balanced (i.e., $r=1/2$).

The following result is the key to prove a Sturm-type separation theorem for the equation $-d(dy/d\alpha)+yd\beta=0$.

\begin{lem}\label{sccv65omn}
Assume that $r$, $\alpha$, and $\beta$ satisfy Hypothesis \ref{H:4.1}. Let $u$ and $v$ be nontrivial real-valued $r$-balanced solutions of $-d(dy/d\alpha)+yd\beta=0$. If there are $s<t$ in $\iOmega$ such that $W^+[u,v](s)W^-[u,v](t)<0$ and $\theta>0$ on $(s,t)$, then $u$ and $v$ are linearly dependent.
\end{lem}

\begin{proof}
The Intermediate Value Theorem \ref{ss3} tells us that for some $x_0\in(s,t)$ $W^+[u,v](x_0)W^-[u,v](x_0)\leq0$. Since $\theta(x_0)>0$ it follows from Equation \eqref{zz32901} that $W^-[u,v](x_0)=0$ which means $u$ and $v$ are linearly dependent.
\end{proof}

The next two theorems are the main results of this section.

\begin{thm}[Sturm Separation Theorem I]\label{sturm0168935}
Assume that $r$, $\alpha$, and $\beta$ satisfy Hypothesis \ref{H:4.1}. Let $u$ and $v$ be linearly independent real-valued $r$-balanced solutions of $-d(dy/d\alpha)+yd\beta=0$ on $\iOmega$. Let $s<t$ be consecutive points in $\mathcal{Z}_{\iOmega}(u)$ such that $\theta>0$ on $(s,t)$. Then
\begin{enumerate}[wide=0pt, leftmargin=19pt, labelwidth=13pt, align=left]
\item If $u^+(s)=u^-(t)=0$, then $v$ changes sign in $(s,t)$.
\item If $u^-(s)=u^-(t)=0$ and $\theta_{r}(s)>0$, then $v$ changes sign in $[s,t)$.
\item If $u^+(s)=u^+(t)=0$ and $\theta_{1-r}(t)>0$, then $v$ changes sign in $(s,t]$.
\item If $u^-(s)=u^+(t)=0$ and $\theta_{r}(s)\theta_{1-r}(t)>0$, then $v$ changes sign in $[s,t]$.
\end{enumerate}
\end{thm}

\begin{proof}
We prove the first statement. Since $s<t$ are consecutive points in $\mathcal{Z}_{\iOmega}(u)$, we may assume that $u^-,u^+>0$ on $(s,t)$. Since $u$ is not the trivial solution, it follows that $(du/d\alpha)^+(s)\neq0$. Also since for all $x\in(s,t)$,
$$0<u^+(x)=\int_{(s,x]}du=\int_{(s,x]}\frac{du}{d\alpha} d\alpha,$$
it follows that $(du/d\alpha)^+(s)>0$ (if $(du/d\alpha)^+(s)<0$, then $du/d\alpha<0$ on $(s,s+\delta)$ for some $\delta>0$). A similar argument shows that $(du/d\alpha)^-(t)<0$. Now assume to the contrary that $v$ does not change sign in $(s,t)$. With the aid of Theorem \ref{ytrop037} we may assume $v>0$ on $(s,t)$. Since $u$ and $v$ are linearly independent it follows that $v^+(s)>0$ and $v^-(t)>0$. Then $W^+[u,v](s)=-v^+(s)(du/d\alpha)^+(s)<0$ and $W^-[u,v](t)=-v^-(t)(du/d\alpha)^-(t)>0$. It follows from Lemma \ref{sccv65omn} that $u$ and $v$ are linearly dependent. This contradiction completes the proof of the first statement.

The proof of the second and the third statements are similar so we only prove the second statement. We may assume that $u^-,u^+>0$ on $(s,t)$. Then $(du/d\alpha)^-(t)<0$. Also, Theorem \ref{ertssiao0} tells us that $\theta_{1-r}(s)(du/d\alpha)^-(s)>0$. Assume to the contrary that $v$ does not change sign in $[s,t)$. We may assume $v>0$ on $[s,t)$. Then $v^-(s)>0$ and $v^-(t)>0$. It follows that $W^-[u,v](t)>0$. On the other hand $W^+[u,v](s)=-\theta_r(s)v^-(s)(du/d\alpha)^-(s)/\theta_{1-r}(s)<0$. A similar contradiction to that in the first proof follows.

Now we prove the forth statement. We may assume $u^-,u^+>0$ on $(s,t)$. Assume to the contrary that $v$ does not change sign in $[s,t]$. Then we may assume $v>0$ on $[s,t]$. Since $v$ does not change sign at $s$ nor $t$ and since $u$ and $v$ are linearly independent it follows that $v^-(s)>0$ and $v^+(t)>0$. With the aid of Theorem \ref{ertssiao0} $\theta_{1-r}(s)(du/d\alpha)^-(s)>0$ and $\theta_{r}(t)(du/d\alpha)^+(t)<0$. Then $W^+[u,v](s)/\theta_r(s)<0$ and $W^-[u,v](t)/\theta_{1-r}(t)>0$. But $\theta_r(s)\theta_{1-r}(t)>0$ so $W^+[u,v](s)W^-[u,v](t)<0$. This contradiction completes the proof.
\end{proof}

We should point out that even if the function $\theta$ does not have the same sign on $(s,t)$, we still can have a separation theorem. This is because $\theta$ changes sign at most at a finite number of points in $(s,t)$. One can check that, if $\theta$ changes sign precisely at an even number of points in $(s,t)$, then Theorem \ref{sturm0168935} is valid.

To prove the second part of the Sturm separation theorem we need the following lemma.

\begin{lem}\label{stroops12}
Assume that $r$, $\alpha$, and $\beta$ satisfy Hypothesis \ref{H:4.1}. Let $u$ and $v$ be linearly independent real-valued $r$-balanced solutions of $-d(dy/d\alpha)+yd\beta=0$ on $\iOmega$. If, for some $s\in\iOmega$, $u^-(s)u^+(s)<0$ and $v$ does not change sign at $s$, then $v(s)(du/d\alpha)(s)W[u,v](s)<0$.
\end{lem}

\begin{proof}
Put $E=v(s)(du/d\alpha)(s)W[u,v](s)$. Observe that $v(s)\neq0$ and $(du/d\alpha)(s)\neq0$. If $(dv/d\alpha)(s)=0$, then $E=-\big(v(s)(du/d\alpha)(s)\big)^2<0$. Assume that $(dv/d\alpha)(s)\neq0$. It follows from Theorem \ref{ertssiao0} that either $v_0\leq -r\Delta_{d\alpha}(s)$ or $v_0\geq (1-r)\Delta_{d\alpha}(s)$, where $v_0=v(s)/(dv/d\alpha)(s)$. Also $-r\Delta_{d\alpha}(s)<u_0<(1-r)\Delta_{d\alpha}(s)$ where $u_0=u(s)/(du/d\alpha)(s)$.
Then
$$E=\big(v(s)\frac{du}{d\alpha}(s)\big)^2\big(\frac{u_0}{v_0}-1\big)<0.$$
This completes the proof.
\end{proof}

\begin{thm}[Sturm Separation Theorem II]
Assume that $r$, $\alpha$, and $\beta$ satisfy Hypothesis \ref{H:4.1}. Let $u$ and $v$ be linearly independent real-valued $r$-balanced solutions of $-d(dy/d\alpha)+yd\beta=0$ on $\iOmega$. Let $s<t$ be consecutive points in $\mathcal{Z}_{\iOmega}(u)$ such that $\theta>0$ on $(s,t)$. Then
\begin{enumerate}[wide=0pt, leftmargin=19pt, labelwidth=13pt, align=left]
\item If $u^-(s)u^+(s)<0$, $u^-(t)=0$, and $\theta_{r}(s)>0$, then $v$ changes sign in $[s,t)$.
\item If $u^+(s)=0$, $u^-(t)u^+(t)<0$, and $\theta_{1-r}(t)>0$, then $v$ changes sign in $(s,t]$.
\item If $u^-(s)u^+(s)<0$, $u^+(t)=0$, and $\theta_{r}(s)\theta_{1-r}(t)>0$, then $v$ changes sign in $[s,t]$.
\item If $u^-(s)=0$, $u^-(t)u^+(t)<0$, and $\theta_{r}(s)\theta_{1-r}(t)>0$, then $v$ changes sign in $[s,t]$.
\item If $u^-(s)u^+(s)<0$, $u^-(t)u^+(t)<0$, and $\theta_{r}(s)\theta_{1-r}(t)>0$, then $v$ changes sign in $[s,t]$.
\end{enumerate}
\end{thm}

\begin{proof}
The first and the second statements have similar proof so we only prove the first statement. We may assume $u^-,u^+>0$ on $(s,t)$. Then $u^-(s)<0$, $u^+(s)>0$, $(du/d\alpha)(s)>0$, and $(du/d\alpha)^-(t)<0$. Assume to the contrary that $v$ does not change sign in $[s,t)$. We may assume $v>0$ on $[s,t)$. Then $W^-[u,v](t)>0$. With the aid of Lemma \ref{stroops12}, $W[u,v](s)<0$. But $\theta_{r}(s)>0$, so $W^+[u,v](s)<0$. It follows from Lemma \ref{sccv65omn} that $u$ and $v$ are linearly dependent. This contradiction completes the proof of the first statement.

The third and the fourth statements have similar proofs so we only prove the third statement. We may assume $u^-,u^+>0$ on $(s,t)$. Then $u^-(s)<0$, $u^+(s)>0$, and $(du/d\alpha)(s)>0$. Since $u$ is not the trivial solution it follows that $(du/d\alpha)^+(t)\neq0$. With the aid of Theorem \ref{ertssiao0}, $\theta_{r}(t)(du/d\alpha)^+(t)<0$. Now assume to the contrary that $v$ does not change sign in $[s,t]$. We may assume $v>0$ on
$[s,t]$. Since $v$ does not change sign at $t$ and since $u$ and $v$ are linearly independent it follows that $v^+(t)>0$. Then $W^-[u,v](t)/\theta_{1-r}(t)>0$. With the aid of Lemma \ref{stroops12}, $W^+[u,v](s)/\theta_{r}(s)=W[u,v](s)<0$. But $\theta_{r}(s)\theta_{1-r}(t)>0$, so $W^-[u,v](t)W^+[u,v](s)<0$. This contradiction completes the proof.

Now we prove the last statement. We may assume $u^-,u^+>0$ on $(s,t)$. Then $(du/d\alpha)(s)>0$ and $(du/d\alpha)(t)<0$. Assume to the contrary that $v$ does not change sign in $[s,t]$. Then we may assume $v>0$ on $[s,t]$. Then Lemma \ref{stroops12} tells us that $W^+[u,v](s)/\theta_{r}(s)=W[u,v](s)<0$ and $W^-[u,v](t)/\theta_{1-r}(t)=W[u,v](t)>0$. But $\theta_{r}(s)\theta_{1-r}(t)>0$ so $W^-[u,v](t)W^+[u,v](s)<0$. This contradiction completes the proof.
\end{proof}

We close this section by deriving a formula for $W^-[u,v]$. Equation \eqref{strj0121} tells us that $W^-[u,v]$ is a left continuous solution of locally bounded variation to the first order differential equation $dy=yd\gamma$, where
$$d\gamma=\frac{(1-2r)\Delta_{d\alpha}}{\theta_{1-r}}d\beta.$$
Now it is straightforward to verify that $1+\Delta_{d\gamma}=\theta$. It follows from Theorem \ref{th01} that any initial value problems consisting of the differential equation $dy=yd\gamma$ has a unique left-continuous solution of locally bounded variation. According to Picard's iterations (see Bennewitz \cite{MR1004432} for a detailed description of Picard's iterations) if $x<y$ in $\iOmega$, then $W^-[u,v](y)=W^-[u,v](x)S(x,y)$, where $S(x,y)$ is given by
$$1+\sum_{t_1}\Delta_{d\gamma}(t_1)+\sum_{t_1<t_2}\Delta_{d\gamma}(t_1)\Delta_{d\gamma}(t_2)+\sum_{t_1<t_2<t_3}\Delta_{d\gamma}(t_1)\Delta_{d\gamma}(t_2)\Delta_{d\gamma}(t_3)+\dots$$
where all indices of the above summations vary over the countably many points in $[x,y)$ where $\Delta_{d\gamma}$ is different from $0$.
This series has the infinite product representation (see problem 3.8.26 in Kaczor and Nowak \cite{MR1751334})
$$\prod_{t\in[x,y)}\theta(t).$$
Therefore if $x<y$ in $\iOmega$, then
$$W^-[u,v](y)=W^-[u,v](x)\prod_{t\in[x,y)}\theta(t).$$

\section{sturm comparison theorem}\label{sct0sct}

In this section our goal is to develop a Sturm-type comparison theorem for balanced solutions ($r=1/2$) of the two differential equations
\begin{equation}\label{rybzp67298}
-d(dy/d\alpha)+yd\beta_1=0
\end{equation}
and
\begin{equation}\label{tyalx76012}
-d(dy/d\alpha)+yd\beta_2=0.
\end{equation}

Before we begin we require the following hypothesis to be satisfied. We introduce the new notation $\omega_j=1-\Delta_{d\alpha}\Delta_{d\beta_j}/4$ for $j=1,2$ that replaces the notation $\theta$.

\begin{hyp}\label{eroq03892}
$\alpha$ is strictly increasing on $\iOmega$, $\beta_1,\beta_2\in\BVl(\iOmega,\mathbb{R})$, and $\omega_1(x)\neq0$ and $\omega_2(x)\neq0$ for all $x\in\iOmega$.
\end{hyp}

Assume that $u$ and $v$ are nontrivial balanced solutions of Equations \eqref{rybzp67298} and \eqref{tyalx76012}, respectively. 
We introduce the Wronskian type quantity
$$\tilde W[v,u]=v\frac{du}{d\alpha}-u\frac{dv}{d\alpha}.$$
Since each of $u$, $v$, $du/d\alpha$, and $dv/d\alpha$ is balanced it follows from Formula \eqref{erybs0341} that
\begin{equation}\label{we76rswa}
d\tilde W[v,u]=uvd(\beta_1-\beta_2).
\end{equation}

The next two comparison theorems are the main results of this section.

\begin{thm}[Sturm Comparison Theorem I]\label{rnzp13ys}
Assume that $\alpha$, $\beta_1$, and $\beta_2$ satisfy Hypothesis \ref{eroq03892}. Let $u$ and $v$ be nontrivial real-valued balanced solutions of Equations \eqref{rybzp67298} and \eqref{tyalx76012}, respectively, on $\iOmega$. Let $s<t$ be consecutive points in $\mathcal{Z}_{\iOmega}(u)$.
\begin{enumerate}[wide=0pt, leftmargin=19pt, labelwidth=13pt, align=left]
\item If $u^+(s)=u^-(t)=0$ and $d(\beta_1-\beta_2)$ is a nontrivial positive measure on $(s,t)$, then $v$ changes sign in $(s,t)$.
\item If $u^-(s)=u^-(t)=0$, $\omega_1(s)>0$, and $d(\beta_1-\beta_2)$ is a positive measure on $[s,t)$, then $v$ changes sign in $[s,t)$.
\item If $u^+(s)=u^+(t)=0$, $\omega_1(t)>0$, and $d(\beta_1-\beta_2)$ is a positive measure on $(s,t]$, then $v$ changes sign in $(s,t]$.
\item If $u^-(s)=u^+(t)=0$, $\omega_1(s)>0$, $\omega_1(t)>0$, and $d(\beta_1-\beta_2)$ is a positive measure on $[s,t]$, then $v$ changes sign in $[s,t]$.
\end{enumerate}
\end{thm}

\begin{proof}
To prove the first statement we may assume that $u^{\mp}>0$ on $(s,t)$. It follows that $(du/d\alpha)^+(s)>0$ and $(du/d\alpha)^-(t)<0$. Assume to the contrary that $v$ does not change sign in $(s,t)$. Then we may assume that $v>0$ on $(s,t)$. It follows that $v^-(t)\geq0$ and $v^+(s)\geq0$. With the aid of \eqref{we76rswa} we obtain
$$0<\int_{(s,t)}uvd(\beta_1-\beta_2)=v^-(t)\Big(\frac{du}{d\alpha}\Big)^-(t)-v^+(s)\Big(\frac{du}{d\alpha}\Big)^+(s)\leq0.$$
This contradiction completes the proof of the first statement.

The second and the third statements have similar proofs so we only prove the second statement. We may assume that $u^{\mp}>0$ on $(s,t)$. Then $u(s)\geq0$ and $(du/d\alpha)^-(t)<0$. Since $\omega_1(s)>0$ it follows from Theorem \ref{ertssiao0} that $(du/d\alpha)^-(s)>0$. Assume to the contrary that $v$ does not change sign in $[s,t)$. Then we may assume that $v>0$ on $[s,t)$. It follows that $v^-(t)\geq0$. Since $\Delta_{d\beta_1}(s)\geq\Delta_{d\beta_2}(s)$ and since $\omega_1(s)>0$ it follows that $\omega_2(s)>0$. Therefore $v^{\mp}(s)>0$.
Thus, by Equation \eqref{we76rswa},
$$0\leq\int_{[s,t)}uvd(\beta_1-\beta_2)=v^-(t)\Big(\frac{du}{d\alpha}\Big)^-(t)-v^-(s)\Big(\frac{du}{d\alpha}\Big)^-(s)<0.$$
This contradiction completes the proof.

To prove the forth statement we may assume $u^{\mp}>0$ on $(s,t)$. Then $u(s)\geq0$ and $u(t)\geq0$. Since $\omega_1(s)>0$ and $\omega_1(t)>0$ it follows from Theorem \ref{ertssiao0} that $(du/d\alpha)^-(s)>0$ and $(du/d\alpha)^+(t)<0$. Assume to the contrary that $v$ does not change sign in $[s,t]$. Then we may assume that $v>0$ on $[s,t]$. Since $\omega_2(s)>0$ and $\omega_2(t)>0$ it follows that $v^-(s)>0$ and $v^+(t)>0$. It follows from Equation \eqref{we76rswa} that
$$0\leq\int_{[s,t]}uvd(\beta_1-\beta_2)=v^+(t)\Big(\frac{du}{d\alpha}\Big)^+(t)-v^-(s)\Big(\frac{du}{d\alpha}\Big)^-(s)<0.$$
This contradiction completes the proof.
\end{proof}

The conclusion of the forth statement in Theorem \ref{rnzp13ys} is valid if $\omega_1(s)<0$, $\omega_1(t)<0$, and $d(\beta_2-\beta_1)$ is a positive measure on $[s,t]$.

To prove the second part of Sturm comparison theorem we need the following result.

\begin{lem}\label{dvzzzop034}
Assume that $\alpha$, $\beta_1$, and $\beta_2$ satisfy Hypothesis \ref{eroq03892}. Let $u$ and $v$ be nontrivial real-valued balanced solutions of Equations \eqref{rybzp67298} and \eqref{tyalx76012}, respectively, on $\iOmega$. If for some $s\in\iOmega$ $v$ does not change sign at $s$ and $u^-(s)u^+(s)<0$, then $v(s)(du/d\alpha)(s)\tilde W[v,u](s)>0$. If, in addition, $\omega_1(s)>0$ and $\Delta_{d\beta_1}(s)\geq\Delta_{d\beta_2}(s)$, then $v(s)(du/d\alpha)(s)\tilde W^{\mp}[v,u](s)>0$.
\end{lem}

\begin{proof}
We begin by pointing out that $v(s)\neq0$ and $(du/d\alpha)(s)\neq0$. If $(dv/d\alpha)(s)=0$, then $v(s)(du/d\alpha)(s)\tilde W[v,u](s)=(v(s)(du/d\alpha)(s))^2>0$. Assume $(dv/d\alpha)(s)\neq0$. Then, with the aid of Theorem \ref{ertssiao0}, $\vert u_0\vert<\Delta_{d\alpha}(s)/2$ and $\vert v_0\vert\geq\Delta_{d\alpha}(s)/2$ where $u_0=u(s)/(du/d\alpha)(s)$ and $v_0=v(s)/(dv/d\alpha)(s)$. It follows that $\vert u_0/v_0\vert<1$.
This completes the proof of the first statement since
$$v(s)\frac{du}{d\alpha}(s)\tilde W[v,u](s)=\big(v(s)\frac{du}{d\alpha}(s)\big)^2\big(1-\frac{u_0}{v_0}\big)>0.$$

Now we prove the second statement. With the aid of Equations \eqref{we7610s200} and \eqref{wg9sur0165} it can be shown that
$$\tilde W^{\mp}[v,u](s)=\omega_1(s)\tilde W[v,u](s){\mp}\frac{1}{2}\big(\Delta_{d\beta_1}(s)-\Delta_{d\beta_2}(s)\big)v(s)\frac{du}{d\alpha}(s)\big(u_0{\mp}\frac{\Delta_{d\alpha}(s)}{2}\big).$$
Since $\omega_1(s)>0$ it follows from the proof of the first statement that $$\omega_1(s)v(s)(du/d\alpha)(s)\tilde W[v,u](s)>0.$$
Also since $\Delta_{d\beta_1}(s)\geq\Delta_{d\beta_2}(s)$ and since $\vert u_0\vert<\Delta_{d\alpha}(s)/2$ it follows that
$${\mp}\frac{1}{2}\big(\Delta_{d\beta_1}(s)-\Delta_{d\beta_2}(s)\big)\big(v(s)\frac{du}{d\alpha}(s)\big)^2\big(u_0{\mp}\frac{\Delta_{d\alpha}(s)}{2}\big)\geq0.$$
This completes the proof.
\end{proof}

\begin{thm}[Sturm Comparison Theorem II]
Assume that $\alpha$, $\beta_1$, and $\beta_2$ satisfy Hypothesis \ref{eroq03892}. Let $u$ and $v$ be nontrivial real-valued balanced solutions of Equations \eqref{rybzp67298} and \eqref{tyalx76012}, respectively, on $\iOmega$. Let $s<t$ be consecutive points in $\mathcal{Z}_{\iOmega}(u)$.
\begin{enumerate}[wide=0pt, leftmargin=19pt, labelwidth=13pt, align=left]
\item If $u^-(s)u^+(s)<0$, $u^-(t)=0$, $\omega_1(s)>0$, and $d(\beta_1-\beta_2)$ is a positive measure on $[s,t)$, then $v$ changes sign in $[s,t)$.
\item If $u^+(s)=0$, $u^-(t)u^+(t)<0$, $\omega_1(t)>0$, and $d(\beta_1-\beta_2)$ is a positive measure on $(s,t]$, then $v$ changes sign in $(s,t]$.
\item If $u^-(s)u^+(s)<0$, $u^+(t)=0$, $\omega_1(s)>0$, $\omega_1(t)>0$, and $d(\beta_1-\beta_2)$ is a positive measure on $[s,t]$, then $v$ changes sign in $[s,t]$.
\item If $u^-(s)=0$, $u^-(t)u^+(t)<0$, $\omega_1(s)>0$, $\omega_1(t)>0$, and $d(\beta_1-\beta_2)$ is a positive measure on $[s,t]$, then $v$ changes sign in $[s,t]$.
\item If $u^-(s)u^+(s)<0$, $u^-(t)u^+(t)<0$, $\omega_1(s)>0$, $\omega_1(t)>0$, and $d(\beta_1-\beta_2)$ is a positive measure on $[s,t]$, then $v$ changes sign in $[s,t]$.
\end{enumerate}
\end{thm}

\begin{proof}
The first and the second statements have similar proofs so we only prove the first statement. We may assume that $u^{\mp}>0$ on $(s,t)$. Then $u>0$ on $(s,t)$, $(du/d\alpha)(s)>0$, and $(du/d\alpha)^-(t)<0$. Assume to the contrary that $v$ does not change sign in $[s,t)$. Then we may assume $v>0$ on $[s,t)$. Then $v^-(t)\geq0$. Since $\omega_1(s)>0$ and since $\Delta_{d\beta_1}(s)\geq\Delta_{d\beta_2}(s)$ it follows from Lemma \ref{dvzzzop034} that $\tilde W^+[v,u](s)>0$. With the aid from Equation \eqref{we76rswa} we obtain
$$0\leq\int_{(s,t)}uvd(\beta_1-\beta_2)=v^-(t)\Big(\frac{du}{d\alpha}\Big)^-(t)-\tilde W^+[v,u](s)<0.$$
This contradiction completes the proof.

The third and the forth statements have similar proofs so we only prove the third statement. We may assume that $u^{\mp}>0$ on $(s,t)$. Then $u>0$ on $(s,t)$ and $u(t)\geq0$. Assume to the contrary that $v$ does not change sign in $[s,t]$. Then we may assume $v>0$ on $[s,t]$. Since $\omega_1(s)>0$ and since $\Delta_{d\beta_1}(s)\geq\Delta_{d\beta_2}(s)$ it follows from Lemma \ref{dvzzzop034} that $\tilde W^+[v,u](s)>0$. Since $\omega_1(t)>0$ it follows from Theorem \ref{ertssiao0} that $(du/d\alpha)^+(t)<0$. With the aid from Equation \eqref{we76rswa} we obtain
$$0\leq\int_{(s,t]}uvd(\beta_1-\beta_2)=v^+(t)\Big(\frac{du}{d\alpha}\Big)^+(t)-\tilde W^+[v,u](s)<0.$$
This contradiction completes the proof.

Now we prove the last statement. We may assume that $u^{\mp}>0$ on $(s,t)$. Then $u>0$ on $(s,t)$, $(du/d\alpha)(s)>0$, and $(du/d\alpha)(t)<0$. Assume to the contrary that $v$ does not change sign in $[s,t]$. Then we may assume $v>0$ on $[s,t]$. Since $\omega_1(s)>0$ and since $\Delta_{d\beta_1}(s)\geq\Delta_{d\beta_2}(s)$ it follows from Lemma \ref{dvzzzop034} that $\tilde W^+[v,u](s)>0$. Since a similar argument works at $t$ it follows that $\tilde W^-[v,u](t)<0$. Then it follows from Equation \eqref{we76rswa} that
$$0\leq\int_{(s,t)}uvd(\beta_1-\beta_2)=\tilde W^-[v,u](t)-\tilde W^+[v,u](s)<0.$$
This contradiction completes the proof.
\end{proof}

\bibliographystyle{plain}

\end{document}